\documentclass[11pt]{amsart}
\usepackage{amsmath,amsthm,amsfonts,amssymb,times,
  latexsym,mathabx,mathrsfs,url}

\numberwithin{equation}{section}  

\DeclareMathAlphabet{\curly}{U}{rsfs}{m}{n}  

\theoremstyle{plain}
\newtheorem{thm}{Theorem}

\newtheorem{lem}{Lemma}

\numberwithin{equation}{section}

\newcommand{\NN}{{\mathbb N}}

\newcommand{\cB}{\ensuremath{\mathcal{B}}}

\newcommand{\cS}{\ensuremath{\mathcal{S}}}

\newcommand{\sP}{\ensuremath{\mathscr{P}}}

\DeclareMathOperator{\lcm}{lcm}

\newcommand{\e}{\ensuremath{\mathrm{e}}}

\makeatletter
\renewcommand{\pmod}[1]{\allowbreak\mkern7mu({\operator@font mod}\,\,#1)}
\makeatother

\newcommand{\be}{\begin{equation}}
\newcommand{\ee}{\end{equation}}

\renewcommand{\le}{\leqslant}

\renewcommand{\ge}{\geqslant}

\newcommand{\ba}{\ensuremath{\mathbf{a}}}

\newcommand{\yildirim}{{Y{\i}ld{\i}r{\i}m}}
\newcommand{\DHL}{{\rm DHL}}
\begin{document}

\title{Solutions of $\phi(n)=\phi(n+k)$ and $\sigma(n)=\sigma(n+k)$}
\author{Kevin Ford}
\date{}
\address{Department of Mathematics, 1409 West Green Street, University
of Illinois at Urbana-Champaign, Urbana, IL 61801, USA}
\email{ford@math.uiuc.edu}

\begin{abstract}
We show that for some even $k\le 3570$ and \emph{all} $k$ with
$442720643463713815200|k$, the equation $\phi(n)=\phi(n+k)$
has infinitely many solutions $n$, where $\phi$
is Euler's totient function. 
We also show that for a positive proportion of all $k$,
the equation $\sigma(n)=\sigma(n+k)$ has infinitely
many solutions $n$.
 The proofs rely on recent
progress on the prime $k$-tuples conjecture by
Zhang, Maynard, Tao and PolyMath.
\end{abstract}

\date{\today}

\thanks{The author is supported in part by 
National Science Foundation grant DMS-1802139.}
\thanks{The author thanks Paul Pollack for 
bringing papers \cite{Erdos74}, \cite{HW} and \cite{PP}
to his attention, Sungjin Kim for sharing 
the preprint \cite{kim}, and Chandra Chekuri for
helpful discussion about graph algorithms.}
\maketitle


\section{Introduction}

We partially solve a longstanding conjecture 
about the solubility of
\be\label{eq:main}
\phi(n+k) = \phi(n),
\ee
where $\phi$ is Euler's function and $k$ is a fixed positive integer.

\smallskip

\noindent
\textbf{Hypothesis $\cS_k$}.  The equation \eqref{eq:main} 
holds for infinitely many $n$.

\smallskip

 Ratat and Goormaghtigh in 1917--18
(see \cite{Dickson-book1}, p. 140)
listed several solutions when $k=1$.
Erd\H os conjectured in 1945 that for any
$m$, the simultaneous equations
\be\label{r-tup}
\phi(n)=\phi(n+1)=\cdots=\phi(n+m-1)
\ee
has infinitely many solutions $n$.  If true, this would immediately
imply hypothesis $\cS_k$ for every $k$.
However, there is only one solution
of \eqref{r-tup} known when $m\ge 3$,
namely $n=5186$, $m=3$.
In 1956, Sierpi\'{n}ski \cite{Sierpinski} showed that
for any $k$, \eqref{eq:main} has at least one solution $n$
(e.g. take $n=(p-1)k$, where $p$ is the smallest prime not dividing $k$).  This was extended by Schinzel \cite{schinzel}
and by Schinzel and Wakulicz \cite{SW}, who showed that
for any $k\le 2\cdot 10^{58}$ there are at least two 
solutions of \eqref{eq:main}.
In 1958, Schinzel \cite{schinzel} explicitly conjectured that
$\cS_k$ is true for every $k\in \NN$.
There is good numerical evidence for $\cS_k$, 
at least when $k=1$ or $k$ is even \cite{B76,B78,BCH,LG,GHP,KKP}.
Details about known solutions when
$k\in \{1,2,3,4,5,6,7,8,9,10,11,12\}$ can also be found in 
OEIS \cite{OEIS} sequences A001274, A001494, A330251, A179186, A179187, A179188, A179189,
A179202, A330429, A276503, A276504 and A217139, respectively.
Below $10^{11}$ there are very few
solutions of \eqref{eq:main} when $k\equiv 3\pmod{6}$
\cite{GHP}, e.g. only the two solutions $n\in \{3,5\}$ for $k=3$ 
are known.  A further search by G. Resta
(see \cite{OEIS}, sequence A330251) reveals
17 more solutions in $[10^{12},10^{15}]$.

\smallskip

There is a close connection between Hypothesis $\cS_k$
for even $k$ and generalized prime twins.

\medskip

\noindent
\textbf{Hypothesis} $\sP(a,b)$: there are infinitely many
$n\in \NN$ such that both $an+1$ an $bn+1$ are prime.

\medskip

Hypothesis $\sP(a,b)$ is believed to be true for any pair
of positive integers $a,b$, indeed this is
a special case of Dickson's Prime $k$-tuples conjecture
\cite{Dickson04}.
Klee \cite{Klee} and Moser \cite{Moser} noted that Hypothesis $\sP(1,2)$
immediately gives $\cS_2$, and Schinzel \cite{schinzel} 
observed that
Hypothesis $\sP(1,2)$ implies $\cS_k$ for every even $k$. 
The proof is simple:
if $n+1$ and $2n+1$ are prime and larger than $k$, then
$\phi(k(2n+1))=\phi((n+1)2k)$.
Graham, Holt and Pomerance \cite{GHP} generalized this idea,
showing the following.

\begin{lem}[{\cite[Theorem 1]{GHP}}]\label{lem:GHP}
For any $k$ and any
number $j$ such that $j$ and $j+k$
have the same prime factors, Hypothesis $\sP(\frac{j}{(j,j+k)},\frac{j+k}{(j,j+k)})$ implies $\cS_k$.
\end{lem}

This also has an easy proof: if  $\frac{j}{(j,j+k)}r+1$ and $\frac{j+k}{(j,j+k)}r+1$
are both prime, then $n=j(\frac{j+k}{(j,j+k)}r+1)$ 
satisfies \eqref{eq:main}.
Note that for odd $k$ there are no such numbers $j$,
 and for each  even $k$  there are
finitely many such $j$ (see \cite{GHP}, Section 3).
Extending a bound of Erd\H{o}s, Pomerance
and S\'{a}rk\H{o}zy \cite{EPS} in the case $k=1$,
Graham, Holt and Pomerance showed that
the solutions of \eqref{eq:main} \emph{not} generated 
from Lemma \ref{lem:GHP} are very rare, with counting function
$O_k(x \exp\{ - (\log x)^{1/3} \})$. 
Pollack, Pomerance and Trevi\~{n}o \cite{PPT} proved a version
uniform in $k$, and
 Yamada \cite{yamada} sharpened this bound to
$O_k(x \exp\{ - (1/\sqrt{2}+o(1))\sqrt{\log x\log\log\log x}\})$.
Assuming the Hardy-Littlewood conjectures \cite{HL23},
when $k$ is even we conclude that there are $\sim C_k x/\log^2 x$
solutions $n\le x$ of \eqref{eq:main}, where $C_k>0$.

At present, Hypothesis $\sP(a,b)$ is not known to hold for any 
pair $a,b$.  However, the work of Zhang, Maynard, Tao and the PolyMath8b
project allow us to conclude $\sP(a,b)$ 
for some pairs $a,b$ from a given collection of pairs.  
To set things up,
we say that a collection of linear forms 
$(a_1n+b_1,\ldots,a_kn+b_k)$ is \emph{admissible}
if $a_i>0$ for each $i$, the forms $a_in+b_i$ are distinct and  
 there is no prime dividing $(a_1n+b_1)\cdots (a_kn+b_k)$
for every integer $n$.

We say that $\DHL^*(k;m)$ holds if,
for any admissible set of $k$ linear forms
$(a_1n+b_1,\ldots,a_kn+b_k)$, 
there exist distinct $i_1,\ldots,i_{m}\in \{1,\ldots,k\}$
such that
there are infinitely many $r$ with the
$m$ numbers $a_{i_1}r+b_{i_1},\ldots,a_{i_{m}}r+b_{i_m}$ simultaneously prime.  
This generalizes the notion $\DHL(k;m)$ from \cite{polymath8b},
which is the analogous statement with the restriction
that $a_1=\cdots=a_k=1$.
Dickson's prime $k$-tuples conjecture implies that
$\DHL^*(k;k)$ is true for every $k$.

\begin{lem}\label{lem:ktup}
$\DHL^*(50;2)$ holds, and for each
$m\ge 3$, there is a constant $k$ so that 
$\DHL^*(k;m)$ holds. 
\end{lem}

The PolyMath8b project \cite[Theorem 16 (i)]{polymath8b}
implies $\DHL(50;2)$, and it requires only trivial 
modifications of the proof to establish $\DHL^*(50;2)$;
we sketch the details in Appendix A.
 Maynard
\cite{Maynard, Maynard-dense} proved the existence of $k$ for any $m$.  
Full details of the proof that $\DHL^*(k;m)$
holds for some $k\ll m e^{4m}$ may be found in the author's
lecture notes \cite{Fsieve}, specifically Section 7 and Theorem 7.6.
This bound slightly improves upon the bound of Maynard
\cite{Maynard}.

  If $a_1,\ldots,a_k$ are distinct, then
the set of forms $(a_1n+1,\ldots,a_kn+1)$ is
always admissible.  Thus, 
given any set $\{a_1,\ldots,a_{50}\}$ of positive integers,
there is an $i\ne j$ so that $\sP(a_i,a_j)$ holds.

\begin{thm}\label{thm: main}
We have
\begin{itemize}
\item[(a)] For any $k$ that is a multiple of
$442720643463713815200$, $\cS_k$ is true;
\item[(b)] There is some even $\ell\le 3570$ such that
$\cS_k$ is true whenever $\ell|k$; consequently,
the number of $k\le x$ for which $\cS_k$ is  true
is at least $x/3570$.
\end{itemize}
\end{thm}

Using Lemma \ref{lem:ktup}, we also
make progress toward Erd\H os' conjecture that
\eqref{r-tup} has infinitely many solutions.

\begin{thm}\label{thm:mtup}
For any $m\ge 3$ there is a tuple of distinct
positive integers $h_1,\ldots,h_m$ so that for any $\ell\in \NN$,
the simultaneous equations
\[
\phi(n+\ell h_1) = \phi(n+\ell h_2) = \cdots = \phi(n+\ell h_m) 
\]
have infinitely many solutions $n$.
\end{thm}

Under the assumption of the Elliott-Halberstam Conjecture (in the notation of \cite{polymath8b}, this is the statement that $\text{EH}[\theta]$ holds for all $\theta<1$),
Maynard \cite{Maynard,Maynard-dense} showed that
$\DHL^*(5;2)$, improving the result $\DHL^*(6;2)$ proved
by Goldston, Pintz and \yildirim\, \cite{GPY}
under the same hypothesis.
A generalized version of the  Elliot-Halberstam conjecture
(in the notation of \cite{polymath8b}, this is the statement that $\text{GEH}[\theta]$ holds for all $\theta<1$), implies
$\DHL(3;2)$ \cite[Theorem 16 (xii)]{polymath8b},
and it appears that that same proof yields
$\DHL^*(3;2)$ (we omit details, as they are not essential to
our results).

\begin{thm}\label{thm:EH}
If $\DHL^*(5;2)$ holds, then $\cS_k$ is true for all $k$ with $30|k$.
If $\DHL^*(4;2)$ holds, then $\cS_k$ is true for all $k$ with $6|k$.
\end{thm}

Incidentally, we cannot improve the second conclusion
of Theorem \ref{thm:EH} by assuming $\DHL^*(3;2)$.

One can ask analogous questions about the sum of divisors 
function $\sigma(n)$.  As $\sigma(p)=p+1$ vs $\phi(p)=p-1$,
oftentimes one can port theorems about $\phi$ over to $\sigma$.
This is not the case here, since our results depend heavily
on the existence of solutions of 
\[
a \phi(b) = b \phi(a),
\]
which is true if and only if $a$ and $b$ have the same set of prime factors.
The analogous equation
\[
a \sigma(b) = b \sigma(a) \;\; \Leftrightarrow \;\; \frac{\sigma(a)}{a} = \frac{\sigma(b)}{b}
\]
has more sporadic solutions, e.g. if $a,b$ are both
perfect numbers or multiply perfect numbers.

\begin{thm}\label{thm:sigma}
For a positive proportion of all $k\in \NN$, the equation
\[
\sigma(n)=\sigma(n+k)
\]
has infinitely many solutions $n$.
\end{thm}

As we shall see from the proof, there is a specific number $A$
and a finite set $\cB$ such that for some element $b\in \cB$,
 the equation $\sigma(n)=\sigma(n+k)$
has infinitely many solutions for all numbers $k=\ell b$
where $(\ell,A)=1$.
Unfortunately, our methods cannot specify any particular
$k$ for which the conclusion holds.  Our method requires
finding, for $t=K_2$, numbers $a_1,\ldots,a_t$ so that
\be\label{sigma-frac}
\frac{\sigma(a_1)}{a_1} = \cdots = \frac{\sigma(a_t)}{a_t} = y.
\ee
Such collections of numbers are sometimes referred to as
``friends'' in the literature, \e.g. \cite{PP}.
Finding larger collections of $a_i$ satisfying \eqref{sigma-frac}
leads to stronger conclusions.

\begin{thm}\label{thm:sigma2}
Let $m\ge 2$, assume $\DHL^*(t;m)$ and assume that there is a $y$ 
and positive integers $a_1,\ldots,a_t$ satisfying \eqref{sigma-frac}.
Then there
are positive integers $h_1<h_2<\cdots < h_m$
so that for a positive proportion of integers $\ell$,
there are infinitely many solutions of
\[
\sigma(n+\ell h_1) = \cdots = \sigma(n+\ell h_m).
\]
\end{thm}

It is known \cite{multi} that for $y=9$,
there is a set of 2095 integers satisfying
\eqref{sigma-frac}.  Thus, since $\DHL(50;2)$ holds,
 Theorem \ref{thm:sigma} follows from the
case $m=2$ of Theorem \ref{thm:sigma2}.
We cannot at present make the conclusion unconditional when $m\ge 3$,
as we do not know that $\DHL^*(2095;3)$ holds.
The best result known in this direction is
$\DHL(35410;3)$
\cite[Theorem 16(ii)]{polymath8b}.

\medskip

\noindent
\textbf{Conjecture A.}
For any $t$, there is an $y$ such that $\sigma(a)/a=y$ has 
at least $t$ solutions.
That is, there are arbitrarily large circles of friends.

\medskip

Clearly, Conjecture A implies the conclusion of Theorem \ref{thm:sigma2} for all $m$.  In \cite{Erdos74}, Erd\H os mentions
Conjecture A and states that he doesn't know of any 
argument that would lead to its resolution.
In the opposite direction, Hornfeck and Wirsing
\cite{HW} showed that for any $y$, there are
$\le z^{o(1)}$ solutions of $\sigma(a)/a=y$ with $a\le z$;
this was improved by Wirsing \cite{W},
who showed that the counting function 
is $O(z^{c/\log\log z})$ for some $c$, uniformly in $y$.
Pollack and Pomerance \cite{PP} studied the solutions
of \eqref{sigma-frac}, gathering data on pairs, triples
and quadruples of friends,
 but did not address Conjecture A.

Using \eqref{sigma-frac} and prime pairs $an-1$ and $bn-1$,
one can generate many solutions of $\sigma(n)=\sigma(n+k)$, analogous to Lemma \ref{lem:GHP}; see Yamada \cite[Theorem 1.1]{yamada}.
For example, one can generate solutions 
with $k=1$ if there is an integer $m$ with
 $\sigma(m)/m=\sigma(m+1)/(m+1)$ (the ratios need
not be integers as claimed in \cite{yamada}). 
If $r>m+1$, and $rm-1$
and $r(m+1)-1$ are both prime, then
$\sigma(m(r(m+1)-1))=\sigma((m+1)(mr-1))$.
 Yamada \cite[Theorem 1.2]{yamada} showed that there are
  $\ll x\exp\{-(1/\sqrt{2}+o(1))\sqrt{\log x\log\log \log x}\}$
solutions $n\le x$ not generated in this way.



{\Large \section{Proofs}}

Throughout, $1\le a < b$ are integers.
We first show that $\sP(a,b)$ implies $\cS_k$ for certain $k$,
inverting Lemma \ref{lem:GHP}.
Define
\be\label{kappa}
\kappa(a,b) = (b'-a')\prod_{p|a'b'} p, \qquad
a'=\frac{a}{(a,b)}, \, b'=\frac{b}{(a,b)}.
\ee
We observe that $\kappa(a,b)$ is always even.

\medskip

\begin{lem}\label{PS}
Assume $\sP(a,b)$.  Then $\cS_k$ holds for every $k$ which is 
a multiple of $\kappa(a,b)$.
\end{lem}

\begin{proof}
Define $a'=\frac{a}{(a,b)}, \, b'=\frac{b}{(a,b)}$
and observe that $\sP(a,b) \Rightarrow \sP(a',b')$.
Let $s=\prod_{p|a'b'} p$, and
suppose that $r>\max(a',b')$ such that $a'r+1$ and $b'r+1$ are both prime.
Let $\ell\in \NN$ and set
\[
m_1 = b'\ell s (a'r+1), \quad m_2 = a'\ell s (b'r+1).
\]
As all of the prime factors of $a'b'$ divide $\ell s$, we have
$\phi(b'\ell s)= b'\phi(\ell s)$ and $\phi(a'\ell s)=a'\phi(\ell s)$,
and it follows than $\phi(m_1)=\phi(m_2)$.  Finally,
$m_1-m_2=(b'-a')\ell s = \ell \kappa(a,b)$.
\end{proof}

\begin{proof}[Proof of Theorem \ref{thm: main}]
Let 
\begin{multline*}
\{a_1,\ldots,a_{50}\} = \{
1, 2, 4, 5, 6, 7, 8, 9, 10, 11, 12, 13, 14, 15, 16, 17, 18, 19, 20, 21, 22, 23, 24, 25, 26, 27, 28, \\29, 30, 31, 32, 33, 34, 35, 36, 37, 38, 39, 40, 41, 42, 43, 44, 45, 46, 47, 48, 49, 52, 56
\},
\end{multline*}
By Lemma \ref{lem:ktup}, for some $i,j$ with $1\le i<j\le 50$,
$\sP(a_i,a_j)$ is true.
We compute
\[
\lcm \{ \kappa(a_i,a_j) : 1\le i<j\le 50 \} =442720643463713815200
=2^5 3^3 5^2 \prod_{7\le p\le 47} p,
\]
and thus (a) follows from  Lemma \ref{PS}.

For part (b), we take
\begin{multline*}
\{a_1,\ldots,a_{50}\} = \{15, 20, 30, 36, 40, 45, 60, 72, 75, 80, 90, 96, 100, 108, 120, 135, 144, 150, 180, 192, 200,\\ 216, 225, 240, 250, 270, 288, 300, 320, 324, 360, 375, 384, 400, 405, 450, 480, 500, 540, 600, \\720, 750, 810, 900, 960, 1080, 1200, 1440, 1500, 1800
\},
\end{multline*}
numbers that only have prime factors $2,3,5$.
We also compute that
\[
\max_{1\le i<j\le 50} \kappa(a_i,a_j) = 3570,
\]
 and 
again invoke Lemma \ref{PS}.
This proves (b).
\end{proof}

\textbf{Remark 1.} 
  For any choice of $a_1,\ldots,a_{50}$,
$\frac{442720643463713815200}{6} | L(\ba)$, where $L(\ba)=\lcm \{\kappa(a_i,a_j): i<j\}$.
Without loss of generality,
assume $(a_1,\ldots,a_{50})=1$.  For a prime $7\le p\le 47$,
 if $p|a_i$ for some $i$
then $p\nmid a_j$ for some $j$ and thus $p|\kappa(a_i,a_j)$.
If $p\nmid a_i$ for all $i$, by the pigeonhole principle,
there are two indices with $a_i\equiv a_j\pmod{p}$.
Again, $p|\kappa(a_i,a_j)$.  Thus, $p|L(\ba)$.
Now we show that $5^2 | L(\ba)$.  Let $S_b=\{a_i : 5^b \| a_i \}$ for $b\ge 0$.  Then $|S_0|\ge 1$.
If $|S_b|\ge 1$ for some $b\ge 2$, then there are $i,j$ with
$5^2|a_i$ and $5\nmid a_j$, and then $5^2|\kappa(a_i,a_j)$.
Otherwise, we have $|S_b|\ge 21$ for some $b\in\{0,1\}$.
By the pigeonhole principle, there is $i\ne j$ with
$5^b\| a_i, 5^b \| a_j$ and $5^{b+2}|(a_i-a_j)$.  This also implies that
$5^2| \kappa(a_i,a_j)$.
Simlarly, let $T_b=\{a_i : 3^b\| a_i \}$.  Then we have 
either  $|T_b|\ge 1$ for some $b\ge 2$, or $|T_i| \ge 7$ for 
some $i\in\{0,1\}$.  Either way, $3^2|L(\ba)$.
Let $U_b=\{ a_i : 2^b\| a_i \}$.  Then either
$|U_b|\ge 1$ for some $b\ge 4$ or $|U_b|\ge 9$ for some
$b\in \{0,1,2,3\}$.  Either way, $2^4 | L(\ba)$. 
It is easy to construct $\ba=(a_1,\ldots,a_{50})$ such that
$3^3 \nmid L(\ba)$ and $2^5 \nmid L(\ba)$.  However, such 
constructions seem to always produce $q|L(\ba)$ for
some prime $q>50$.

\medskip

\textbf{Remark 2.}
We believe that 3570 is the smallest number that can be produced 
for Theorem \ref{thm: main} (b) using $\DHL^*(50;2)$,
however we do not have a proof of this.
We did perform a sophisticated search based on graph algorithms.
We limited our
search to sets of numbers composed only of the primes 2,3,5,
as additional prime factors  always seem to produce a
value of $\kappa(a,b)$ larger than 3570.
  For a given finite set of 
integers $\{b_1,\ldots,b_r\}$, the problem of minimizing $\max_{i,j\in I} \kappa(b_i,b_j)$
over all 50-element subsets $I \subset \{1,\ldots,r\}$, is equivalent to
that of finding the largest clique in a graph.  Indeed, take 
a threshold value $t$,
vertex set
$\{1,\ldots,r\}$ and draw an edge from $i$ to $j$ if $\kappa(b_i,b_j)\le t$.  Then the graph has a 50-element clique if and only
if $\max_{i,j\in I} \kappa(b_i,b_j) \le t$.
Using the Sage routing \texttt{clique\_number()} with $t=3569$
and  $\{b_1,\ldots,b_r\}$ being the smallest 800 numbers composed only 
of primes 2,3,5 (the largest being 12754584), we found that the largest clique has size 49.

\medskip

\textbf{Remark 3.}
The author recently learned that Sungjin Kim \cite{kim}
proved weaker statements in the direction of
Theorem \ref{thm: main}.
He used Lemma \ref{lem:ktup} to show that $\cS_k$ holds for
some $k\in \{B,2B,\ldots,50B\}$, with
$B=\prod_{p\le 50} p$, and
that the set of $k$ for which $\cS_k$ holds
has counting function $\gg\log\log x$.

Andrew Granville informed the author that Chris Orr has independently
discovered Lemma \ref{PS} and the conclusion that there is some $k$
for which $\cS_k$ holds.  The author has not seen the paper.

\begin{proof}[Proof of Theorem \ref{thm:mtup}]
Let $m\ge 2$, and let $k$ be such that
$\DHL^*(k;m)$ holds.  The existence of such $k$ 
follows from Lemma \ref{lem:ktup}.
Consider any set 
$\{a_1,a_2,\ldots,a_k\}$ of $k$ positive integers.  
Then
there are $1\le i_1 < i_2 <\cdots < i_m\le k$ 
such that for infinitely many $r$,
the $m$ numbers $a_{i_1} r +1,\ldots,a_{i_m} r + 1$ are all prime.
Let $r$ be such a number.
Define
\[
h_j = \frac{(a_{i_1}\cdots a_{i_m})^2}{a_{i_j}} \qquad (1\le j\le m).
\]
Let $\ell\in \NN$ and set
 $n=\ell (a_{i_1}\cdots a_{i_m})^2 r$.  Then, since
$a_{i_j}|h_j$ for all $j$, it follows that
for any $j$,
\[
\phi(n+\ell h_j)=\phi(\ell h_j(a_{i_j}r+1))=
\phi(\ell h_j) a_{i_j} r=\phi(\ell h_j a_{i_j}) r.\qedhere
\]
\end{proof}

\begin{proof}[Proof of Theorem \ref{thm:EH}]
Same as the proof of Theorem \ref{thm: main} (a), but take $\{a_1,a_2,a_3,a_4,a_5\}=\{1,2,3,4,6\}$ if $\DHL^*(5;2)$ holds
and $\{a_1,\ldots,a_4\}=\{1,2,3,4\}$ if $\DHL^*(4;2)$ holds.
\end{proof}

\begin{proof}[Proof of Theorem \ref{thm:sigma2}]
Assume $\DHL^*(t;m)$ and suppose that $a_1,\ldots,a_t$ satisfy \eqref{sigma-frac}.
Put $A=\lcm[a_1,\ldots,a_t]$ and for each $i$ define
$b_i = A/a_i$.
By Lemma \ref{lem:ktup} applied to the collection of 
linear forms $b_in-1$, $1\le i\le t$, there exist
$i_1,\ldots,i_m$ such that for infinitely many $r\in \NN$,
the $m$ numbers $b_{i_j}r-1$ are all prime.
Let $r>A$ be such a number,
and let $\ell\in \NN$ such that $(\ell,A)=1$ (this holds for a
positive proportion of all $\ell$).
Let 
\[
t_j = \ell a_{i_j} (b_{i_j}r-1) = A\ell r - \ell a_{i_j} \qquad (1\le j\le m).
\]
By \eqref{sigma-frac}, for every $j$ we have
\[
\sigma(t_j) = \sigma(\ell) \sigma(a_{i_j}) b_{i_j}r = y \sigma(\ell) Ar.\qedhere
\]
\end{proof}

\appendix

\section{Details of $\DHL^*(50;2)$}

The  statement $\DHL^*(50;2)$
 follows immediately from an
appropriate generalization of Theorem 26 in \cite{polymath8b} to arbitrary linear forms
$a_i n+b_i$, together with the calculations
given in Theorems 27 (Theorem 27 requires no modification
as it concerns the existence of smooth functions in polytopes with
certain properties).
These in turn depend on a generalization of the
simple combinatorial Lemma 18 (the modification is trivial),
plus generalizations of
 Theorem 19 (i) (but not part (ii)!) and Theorem 20 (i) (again, only part (i) not part (ii)).  The proof of the latter two Theorems is given
 on pages 17--23 in \cite{polymath8b}, the one
 needs only replace the forms $n+h_i$ (in the notation
 of \cite{polymath8b}) with the forms $a_in+b_i$
 throughout.  With slightly different notation,
 details may be found in Section 7 of the author's 
 lecture notes \cite{Fsieve}.


\end{document}